\documentclass[article]{article} 

\usepackage{mathtools}
\usepackage{amsmath}
\usepackage{cases}
\usepackage{float}

\usepackage{kantlipsum}
\usepackage{tikz}
\usepackage{pgf,tikz}
\usetikzlibrary{arrows}
\usetikzlibrary[patterns]

\usepackage{amsmath, amssymb, amsthm}
\usepackage{color}

\theoremstyle{plain}
\newtheorem{thm}{Theorem}[section]

\newtheorem*{prop*}{Proposition}
\newtheorem{lem}[thm]{Lemma}
\newtheorem*{lem*}{Lemma}

\newtheorem*{rem}{Remark}

\newcommand{\beq}{\begin{equation}}
\newcommand{\eeq}{\end{equation}}
\date{}
\title{\vspace{-0.7cm} {On the number of cliques in graphs with a forbidden minor}}

\author{
Jacob Fox\thanks{Department of Mathematics, Stanford University, Stanford, CA 94305. Email: {\tt jacobfox@stanford.edu}. Research supported by a Packard Fellowship, by NSF Career Award DMS-1352121 and by an Alfred P. Sloan Fellowship.}
\and
Fan Wei\thanks{Department of Mathematics, Stanford University, Stanford, CA 94305. Email: {\tt fanwei@stanford.edu}.}
}

\begin{document}

\maketitle

\begin{abstract}

Reed and Wood and independently Norine, Seymour, Thomas, and Wollan proved that for each positive integer $t$ there is a constant $c(t)$ such that every graph on $n$ vertices with no $K_t$-minor has at most $c(t)n$ cliques. Wood asked in 2007 if we can take  $c(t) = c^t$ for some absolute constant $c$. This question was recently answered affirmatively by Lee and Oum. In this paper, we determine the exponential constant. We prove that every graph on $n$ vertices with no $K_t$-minor has at most $3^{2t/3+o(t)}n$ cliques. This bound is tight for $n \geq 4t/3$. More generally, let $H$ be a connected graph on $t$ vertices, and $x$ denote the size (i.e., the number edges) of the largest matching in the complement of $H$. We prove that every graph on $n$ vertices with no $H$-minor has at most $\max(3^{2t/3-x/3+o(t)}n,2^{t+o(t)}n)$ cliques, and this bound is tight for $n \geq \max (4t/3-2x/3,t)$ by a simple construction. Even more generally, we determine explicitly the exponential constant for the maximum number of cliques an $n$-vertex graph can have in a minor-closed family of graphs which is closed under disjoint union. 
\end{abstract}

\section{Introduction} \label{sec:intro}

Tur\'an's theorem \cite{Tu} is a cornerstone result in extremal graph theory. It determines the maximum number of edges  a $K_t$-free graph on $n$ vertices can have, and it is given by the balanced complete $(t-1)$-partite graph. This was extended by Zykov \cite{Zy}, who determined that the same graph maximizes the number of cliques amongst $K_t$-free graphs on $n$ vertices. 

Analogous questions for minors have also been studied for a long time. A graph $H$ is a {\it minor} of a graph $G$ if it can be obtained from $G$ by contracting edges and deleting vertices and edges. Mader \cite{Ma,Ma1} proved that for each positive integer $t$ there is a constant $C(t)$ such that every graph on $n$ vertices with no $K_t$-minor has at most $C(t)n$ edges.  Kostochka \cite{Ko,Ko1} and Thomason \cite{Th} independently proved that $C(t)=\Theta(t\sqrt{\log t})$.  In particular, every graph with no $K_t$-minor has a vertex of degree $O(t\sqrt{\log t})$. Thomason \cite{Th} later proved that $C(t)=(\alpha+o(1))t\sqrt{\log t}$ where $\alpha=0.319...$ is an explicit constant.

Norine, Seymour, Thomas, and Wollan \cite{NSTW} and independently Reed and Wood \cite{RW} showed that for each $t$ there is a constant $c(t)$ such that every graph on $n$ vertices with no $K_t$-minor has at most $c(t)n$ cliques. Wood \cite{Wo} asked in 2007 if $c(t)<c^t$ for some absolute constant $c$. Progress on this question was made in \cite{FOT}, and it was recently resolved by Lee and Oum \cite{LO}. They proved that every graph on $n$ vertices with no $K_t$-subdivision (and hence  every graph on $n$ vertices with no $K_t$-minor) has at most $2^{5t+o(t)}n$ cliques, and observed that optimizing their proof would improve the exponential constant from $5$ to a number less than $4$. Wood \cite{Wo1} very recently determined, for $t \leq 9$ and $n \geq t-2$, that the maximum number of cliques a $K_t$-minor free graph on $n$ vertices can have is $2^{t-2}(n-t+3)$. He further conjectures that this bound holds if and only if $t \leq 49$. It may be surprising that there is a simple construction that is optimal up to such a large value of $t$ before another construction does significantly better. For $t$ large enough, a construction of a $K_t$-minor free graph on $n$ vertices with more cliques is the disjoint union of graphs on $2\lceil 2t/3 \rceil -2$ vertices which are the complement of a perfect matching. This graph has $3^{2t/3 - o(t)}n \geq 2^{1.0566t -o(t)}n$ cliques. 

It is not surprising that this natural question in extremal graph theory, bounding the number of cliques a graph on a given number of vertices with a forbidden minor can have, has several significant algorithmic applications. These include linear-time algorithms for testing first-order properties in sparse graphs \cite{DKT}, and in a fast algorithm for finding small separators in graphs with a forbidden minor \cite{RW}, which in turn has many further algorithmic applications. 

In this paper, we determine the exponential constant in Wood's problem. 

\begin{thm} \label{main} Every graph on $n$ vertices with no $K_t$-minor has at most $3^{2t/3+o(t)}n$ cliques. This bound is tight for $n \geq 4t/3$.
\end{thm} 

The proof of Theorem \ref{main} is given in Section \ref{minorsec} and has a few ingredients. In Subsection \ref{cliquenumber}, we prove a lemma which shows that for every graph of very large minimum degree, the order of the largest clique minor is about the average of its number of vertices and the order of its largest clique. In Subsection \ref{cliquenumber2}, we give a tight upper bound on the number of cliques in a graph with given sum of the number of vertices and the order of its largest clique. We use these lemmas to prove Theorem \ref{main} in Subsection \ref{proofofmain}.

Our tools also allow us to estimate the maximum possible number of cliques in the case $n<4t/3$. Of course, when $n<t$, any graph with $n$ vertices has no $K_t$-minor and hence the maximum possible number of cliques is just $2^{n}$ in this case, when the graph is a clique. When $t \leq n < 4t/3$, the maximum number of cliques a $K_t$-minor free graph on $n$ vertices can have is $2^{4t-3n+2(n-t)\log_2 3 + o(t)}$. Tightness for $n \leq \frac{4t-2}{3}$ is given by the graph on $n$ vertices which is the complement of a matching of size $x=2(n-t)+1$. It is straightforward to check that this graph has no $K_t$-minor and has $3^x2^{n-2x}$ cliques, giving the lower bound. The proof of the upper bound is given in Theorem \ref{main7}. 

With some alternative tools, including probabilistic and linear programming methods, we generalize Theorem \ref{main} and prove the following theorem which determines the exponential constant for the maximum number of cliques a graph on $n$ vertices with no $H$-minor can have. For a graph $H$, let $x(H)$ denote the number of edges of the largest matching in the complement of $H$. We also sometimes refer to this as the maximum missing matching size in $H$. 

\begin{thm}\label{generalmain}
Let $H$ be a connected graph on $t$ vertices and $x=x(H)$. Every graph on $n$ vertices with no $H$-minor has at most $\max (3^{2t/3- x/3+o(t)}n, 2^{t + o(t)}n )$ cliques. This is tight for $n \geq \max (4t/3-2x/3,t)$.
\end{thm}
Theorem \ref{generalmain}, when $x \leq (2-3/\log 3)t$, claims that every graph on $n \geq 4t/3-2x/3$ vertices with no $H$-minor has at most $3^{2t/3- x/3+o(t)}n$ cliques. This bound is tight by considering a disjoint union of copies of the graph which is the complement of a perfect matching on $2\lceil 2t/3-x/3\rceil -2$ vertices. Theorem \ref{generalmain}, when $x >  (2-3/\log 3)t$, claims that every graph on $n$ vertices with no $H$-minor has at most $2^{t+o(t)}n$ cliques. The exponent is sharp by the construction of a graph of $n$ vertices which is the disjoint union of cliques on $t-1$ vertices with possibly one clique on fewer vertices. 

We further extend this result from a graph with a single forbidden minor to a graph in a minor-closed family of graphs. Recall the celebrated Robertson-Seymour Graph Minor Theorem \cite{RS}, which states that for any minor-closed family $\mathcal{G}$ of graphs, the set $\mathcal{H}(\mathcal{G})$ of minimal forbidden minors is finite. 
We want to bound the number of cliques a graph in $\mathcal{G}$ on $n$ vertices can have. It is easy to check that a minor-closed family $\mathcal{G}$ is closed under disjoint union if and only if the minimal forbidden minors are all connected graphs.

\begin{thm}\label{family}
Let $\mathcal{G}$ be a minor-closed family of graphs which is closed under disjoint union, and $\mathcal{H}(\mathcal{G}) = \{ H_1, H_2, \dots, H_{\ell}\}$ be the set of minimal forbidden minors. Suppose for $1 \leq i \leq \ell$ that  $H_i$ has $t_i$ vertices and its maximum missing matching size is $x_i$. Assume $t_1 \leq t_2 \leq \cdots \leq t_{\ell}$.  
Then the number of cliques a graph in $\mathcal{G}$ on $n$ vertices can have is at most
\beq 
\max_{(a_j, b_j)~\text{an extreme point of the lower envelope $\mathcal{P}$}} 2^{(\log_2 3)a_j +b_j + o(t_1)}n,
\eeq
where the lower envelope $\mathcal{P}$ is defined in Section \ref{secfamily} and the extreme points of $\mathcal{P}$ are easy to compute and there are only a finite number of them. 
The lower order term $o(t_1)$ tends to $0$ as $t_1 \to \infty$. This bound is tight for $n \geq \frac{4}{3}t_1$. 
\end{thm}

Tightness for Theorem \ref{family} comes from the disjoint union of copies of a graph which is a complement of a (not necessarily perfect) matching, i.e., of a graph whose complement has maximum degree at most one. 

\vspace{0.2cm}

\noindent {\bf Organization:} In the next section we prove Theorem \ref{main}. We use different tools to prove the generalization Theorem \ref{generalmain} in Section \ref{section3}, and we further generalize this result to Theorem \ref{family} in Section \ref{secfamily}.  We finish with some concluding remarks. For the sake of clarity of presentation, we sometimes omit floor and ceiling signs. Throughout the paper, unless specified otherwise, all logarithms are base two.  We also do not make any serious attempts to optimize the $o(t)$ contribution that appear in the exponents.

\section{Counting cliques in graphs with no $K_t$-minor}\label{minorsec}

In the first subsection, we prove a lemma showing that the order of the largest clique minor in a very dense graph is the greatest integer which is at most the average of the number of vertices and the order of its largest clique. In the following subsection, we prove that the maximum number of cliques in a graph given the sum of the number of vertices and the order of its largest clique is obtained by the complement of a perfect matching. We finish with a subsection proving Theorem \ref{main} by a reduction to the very dense case and using the lemmas proved in the first two subsections. 

\subsection{Hadwiger number of very dense graphs}\label{cliquenumber}

The Hadwiger number $h(G)$ of a graph $G$ is the order of the largest clique minor in $G$. The clique number $\omega(G)$ is the order the largest clique in $G$. Let $G$ be a graph on $n$ vertices with clique number $\omega$ and Hadwiger number $h$. A simple upper bound on the Hadwiger number is $h \leq \lfloor \frac{n+\omega}{2}\rfloor$.  Indeed, if we form a clique minor of size $h$ by contracting sets, and we use $n_i$ sets of size $i$, then the $n_1$ vertices form a clique, $h = \sum_i n_i$, and the number of vertices used is $\sum_i in_i$, which is at most $n$. Hence, $2h \leq n+n_1$, from which we get $h \leq (n+n_1)/2 \leq (n+\omega)/2$. As $h$ is an integer,  we have $h \leq \lfloor \frac{n+\omega}{2} \rfloor$. The next lemma shows that the Hadwiger number matches this upper bound if the graph has large minimum degree and does not have a nearly spanning clique. 

\begin{lem}\label{mainlem1}
Let $G$ be a graph on $n$ vertices with minimum degree $\delta$ and clique number $\omega$. Let $\Delta=n-\delta-1$, which is the maximum degree of the complement of $G$. If $n \geq \omega+2\Delta^2+2$, then $h(G)=\lfloor \frac{n+\omega}{2}\rfloor$.
\end{lem}
\begin{proof}

As the upper bound was established before the lemma, it suffices to produce a clique minor of order $\lfloor \frac{n+\omega}{2}\rfloor$.  We will build a clique minor by using a maximum clique $K$ and carefully pairing up all the remaining vertices (apart from possibly one vertex due to parity) to create the clique minor. We will do a little more, as the perfect matching in $V(G) \setminus K$ (we later show the perfect matching exists) we use for the pairs are such that each edge in the perfect matching is a dominating set in $G$. 

We make an auxiliary graph $A$ on $V(G)$, where a pair $u,v$ of vertices are adjacent in $A$ if they are adjacent in $G$ and form a dominating set, i.e., every vertex of $G$ is adjacent to $u$ or $v$. Note that the edges of $A$ are precisely those pairs of vertices of distance at least three in the complement of $G$. Succintly, $A$ is the complement of the graph which is the square of the complement of $G$. As the complement of $G$ has maximum degree $\Delta$, the complement of $A$ has maximum degree at most $\Delta+\Delta(\Delta-1)=\Delta^2$. Let $n'=n-\omega$, which is the number of vertices of $G$ not in $K$. Recall Dirac's theorem \cite{Dir} that every graph with minimum degree at least half the number of vertices has a hamiltonian cycle and hence a perfect matching (apart from one vertex, if the number of vertices is odd). Therefore, if the minimum degree of the induced subgraph $A\setminus K$ of $A$ on $V(G) \setminus K$ is at least $n'/2$, then, $A\setminus K$ has a perfect matching (apart from possible a single vertex due to parity), which would complete the proof. So it suffices to check that $n'-\Delta^2-1$, which is a lower bound on the minimum degree of $A\setminus K$, is at least $n'/2$. This inequality is equivalent to the condition in the lemma statement.

\end{proof}

\subsection{On the number of cliques and the clique number}\label{cliquenumber2}

For a graph $G$, let $c(G)$ denote the number of cliques in $G$, $n(G)$ denote the number of
vertices of $G$, and $\omega(G)$ be the clique number of $G$. Define $k(s)$ to be the maximum of $c(G)$ over all graphs $G$ with $n(G)+\omega(G) \leq s$. 

\begin{lem}\label{mainlem2} We have $k(s) \leq 3^{s/3}$, with equality when $s$ is a multiple of $3$
and $G$ is the complement of a perfect matching. 
\end{lem}

\begin{proof} The proof is by induction on $s$, the base case $s < 3$ being trivial. Let
$G$ be a graph with $k(s)$ cliques, $n$ vertices, clique number $\omega$, and $n+\omega \leq s$. 

\noindent {\bf Case 1:} The complement of $G$ has a vertex $v$ of degree $0$. \\ Deleting $v$ decreases the number
of vertices by one and the clique number by one, and halves the number of cliques. We thus get $$k(s) \leq 2k(s-2) \leq 2 \cdot 3^{(s-2)/3} = 2\cdot 3^{-2/3}3^{s/3} < 3^{s/3}$$ by the induction hypothesis, completing this case. 

\noindent {\bf Case 2:} The complement of $G$ has a vertex $u$ with degree $d$ at least two. \\ Every clique of $G$ not containing $u$ has one less vertex, giving at most $k(s-1)$ such cliques. Every clique of $G$ containing $u$ is determined by the clique apart from $u$ and its $d$ non-neighbors, decreasing the number of possible
vertices by $d+1$ and the clique number by $1$, giving at most $k(s-d-2)$ such cliques. 
We thus obtain in this case $$k(s) \leq k(s-1)+k(s-d-2) \leq 
k(s-1)+k(s-4),$$ and, by the induction hypothesis, this is at most
$$3^{(s-1)/3}+3^{(s-4)/3}=3^{s/3}(3^{-1/3}+3^{-4/3}) < 3^{s/3},$$ completing this case. 

\noindent {\bf Case 3:} Every vertex has degree one in the complement. \\ The graph is
thus the complement of a perfect matching. Then $n=2s/3$ and $\omega=s/3$ and $G$ has
$3^{s/3}$ cliques, as we can choose independently for each non-edge of $G$ three possible intersections for that pair with the clique. 

\end{proof}

\subsection{Proof of Theorem \ref{main}} \label{proofofmain}

The goal of this section is to prove Theorem \ref{main}. We assume that $G$ is a graph with $n$ vertices and no $K_t$-minor.  

Lee and Oum \cite{LO} (and in earlier works such as \cite{KW}) present a simple algorithm to enumerate all the cliques in a graph $G$, called the ``\emph{peeling process}". It is helpful in bounding the number of cliques in a graph. Since we will use this peeling process again later in the paper, we separate it here from the rest of the text.

\vspace{0.3cm}

\noindent {\bf Peeling Process}  

Pick a minimum degree vertex $v_1$ and include it in the clique, and continue the algorithm picking cliques to add to $v_1$ within the neighborhood $N(v_1)$ of $v_1$. Once we have exhausted all cliques containing $v_1$, we delete it from the graph and continue the algorithm amongst the remaining vertices. 


In this way, every clique $K$ in $G$ has an ordering $v_1,\ldots,v_s$ of all its vertices so that the following holds. Let $G_0=G$. After deleting vertices one at a time of degree smaller than the degree of $v_1$, we obtain an induced subgraph $G_1$ that contains $K$ in which $v_1$ has the minimum degree.  After picking $v_1,\ldots,v_i$, we delete from $G_i$ vertex $v_i$ and its nonneighbors, and vertices one at a time of degree smaller than that of $v_{i+1}$ and obtain an induced subgraph $G_{i+1}$ of $G_i$ that contains $K \setminus \{v_1,\ldots,v_i\}$ and in which $v_{i+1}$ has the minimum degree. 

The peeling process allows us to focus on a remaining induced subgraph which is quite dense and is easier to analyze. There are not many vertices of any given clique not in a remaining dense induced subgraph. These few vertices contribute a relatively small factor to the total number of cliques.


\begin{proof}[Proof of Theorem \ref{main}]
We bound the number of cliques by the peeling process. Let $n_i$ denote the number of vertices in $G_i$.   Let $c=0.1$ and $r=r(K)$ be the least positive integer such that $n_r \leq 1.05t$ or $n_{r+1} \geq n_{r}-cn_{r}^{1/2}$ or $r=s$. Recall that $s$ is the order of clique $K$. 

We first give a bound on $r$. The result of Thomason \cite{Th1} implies, as $G$ does not contain a $K_t$-minor, that every subgraph of it has a vertex of degree at most $d:=t\sqrt{\log t}$ (here we assume $t$ is sufficiently large). Hence $n_1 \leq d+1$. We have $n_i < n_{i-1}-cn_{i-1}^{1/2}$ for each $i \leq r$. In particular, if $j \leq r$ and $j-i \geq c^{-1}(2n_i)^{1/2}$, then $n_j<n_i/2$. It follows that $$r \leq 1+\sum_{h \geq 0} 1+ c^{-1}((d+1)/2^{h-1})^{1/2} \leq 4c^{-1}t^{1/2}\log^{1/4} t:=r_0.$$

We next give a bound on the number of choices for the first $r$ vertices $v_1,\ldots,v_r$ in cliques. We have $n_0=n$ choices for $v_1$. We will use a weak estimate that the number of choices for $v_2,\ldots,v_r$ having picked $v_1$ is at most 
$${|G_1|-1 \choose \leq r_0} \leq r_0 \binom{t\sqrt{\log t}}{r_0} \leq r_0 \left(\frac{e t \sqrt{\log t}}{r_0}\right)^{r_0} =  2^{O(t^{1/2}\log^{5/4} t)}.$$
We thus have at most $n2^{O(t^{1/2}\log^{5/4} t)}$ choices for $v_1,\ldots,v_r$. 

Recall that $G$ has $n$ vertices and no $K_t$-minor and our goal is to bound the number of cliques in $G$. We have already bounded the number of choices for the first $r$ vertices in a clique, and it suffices to bound the number of choices for the remaining vertices. We split the cliques into three types: those with $n_r \leq 1.05t$, those with $r=s$ and $n_r > 1.05t$, and those with $r<s$, $n_r>1.05t$, and $n_{r+1} \geq n_r-cn_r^{1/2}$. 



We first bound the number of cliques with $n_r \leq 1.05t$. As there are at most $1.05t$ possible remaining vertices to include after picking $v_1,\ldots,v_r$ for the clique, then there are at most $2^{1.05t}$ ways to extend these vertices. We thus get at most $n2^{1.05t+O(t^{1/2}\log^{5/4} t)}$ cliques of the first type. 

We next bound the number of cliques with $r=s$. We saw that this is at most $2^{r_0}=2^{O(t^{1/2}\log^{5/4} t)}$.

Finally, we bound the number of cliques with  $n_r \geq 1.05t$, $r<s$, and $n_{r+1} \geq n_r-cn_r^{1/2}$. In this case, we are left with a graph $G_r$ where $v_r$ has the minimum degree, and it and its non-neighbors are not in $G_{r+1}$, which has $n_{r+1}$ vertices. Thus, the complement of $G_r$ has maximum degree $\Delta \leq n_{r}-n_{r+1} \leq cn_r^{1/2}$. Thus $G_r$  is dense. Let $\omega$ be the clique number of $G_r$. Since $G$ has no $K_t$-minor, $\omega < t$, so $n_r-\omega > n_r-t\geq n_r - \frac{n_r}{1.05} \geq 0.04n_r \geq 2(cn_r^{1/2})^2+2 \geq 2\Delta^2+2$, where we used that $t$ is sufficiently large. Hence, the condition of Lemma \ref{mainlem1} is satisfied, so the Hadwiger number $h(G_r)$ is $\lfloor \frac{n_r+\omega}{2} \rfloor<t$, and so $n_r+\omega<2t$. By Lemma \ref{mainlem2}, the number of cliques in $G_r$ is at most $3^{2t/3}$. This gives a bound on the number of ways of completing a clique in $G$ having picked the first $r$ vertices. We thus get at most $3^{2t/3+O(t^{1/2}\log^{5/4} t)}n$ cliques of the last type. 

Adding up all possible cliques, we get at most $3^{2t/3+O(t^{1/2}\log^{5/4} t)}n$ cliques in $G$, completing the proof. 
\end{proof}

\section{Counting cliques in graphs with no $H$-minor}\label{section3}
In this section we prove Theorem \ref{generalmain} which, for any connected graph $H$, gives an essentially sharp upper bound on the number of cliques an $H$-minor free graph on $n$ vertices can have. 

By the peeling process as described in Subsection \ref{proofofmain} and similar to the proof of Theorem \ref{main}, we end up with a very dense induced subgraph $G_r$ where the maximum missing degree $\Delta$ is very small. We can then pass to a minor of this induced subgraph with the most cliques, and this minor will also be $H$-minor free. This minor is an example of a social graph, which we next define. 

A {\it social graph} is a graph such that contracting any edge of it does not increase the number of cliques. It appears to be a challenging problem to characterize social graphs. We show in the next subsection that any very dense social graph $G$ is, apart from a small number of vertices, the complement of a (not necessarily perfect) matching. This is an important tool in our results bounding the number of cliques in $H$-minor free graphs and more generally in minor-closed families of graphs which are closed under disjoint union. The proof of this important tool is shown in two steps. First, we show that almost every vertex is in a large proportion of the cliques in $G$, i.e., the fraction of cliques in $G$ containing that vertex is large. We call those vertices {\it good} and will define this notion formally soon. Second, we show that by ignoring the small number of vertices which are not good and their neighbors, each remaining vertex has non-degree at most one.

\subsection{The structure of very dense social graphs} \label{structure}
We now formally define what it means for a vertex to be good. 
For each vertex $v \in G$, let $\alpha_v$ denotes the fraction of cliques in $G$ containing $v$. Let $\alpha^*$ be a number slightly smaller than $\frac{2-\sqrt{2}}{2}$, say $\alpha^* = \frac{2-\sqrt{2}}{2} - 0.001 \approx .29189$. We say $v$ is a \emph{good vertex} if $\alpha_v \geq \alpha^*$, and otherwise it is a \emph{bad vertex}. Let $U$ be the set of bad vertices in $G$. We next show that if $G$ contains too many bad vertices, then we can always find two adjacent bad vertices to contract such that the total number of cliques strictly increases. Therefore a dense social graph cannot have too many bad vertices.

The next lemma bounds the number of bad vertices in a social graph in terms of the maximum missing degree (i.e., the maximum degree in the complement). 

\begin{lem}\label{contract}
Every social graph $G$ with maximum missing degree $\Delta$ has at most $300\Delta^2$ bad vertices. 
\end{lem}

\begin{proof}

We suppose for contradiction that the number $|U|$ of bad vertices is at least $300\Delta^2$. We may assume $\Delta \geq 1$ since otherwise $G$ is a clique and there are no bad vertices. 
Since $G$ has maximum missing degree $\Delta$, then, in the complement of $G$,  each vertex has distance at most two to at most $\Delta+\Delta(\Delta-1)=\Delta^2$ other vertices. Thus we can greedily construct a subset $U'$ of the set $U$ of bad vertices of cardinality at least $|U|/(\Delta^2+1) \geq 300\Delta^2 /(\Delta^2+1)  \geq  150$ such that every pair of vertices in $U'$ have distance at least three in the complement of $G$. Equivalently, this condition says that $U'$ forms a clique in $G$ and every pair of vertices in $U'$ is a dominating set of $G$. 


Pick a clique $K$ in $G$ uniformly at random. By linearity of expectation, we have 
\beq E[|K \cap U'|] = \sum_{u \in U'} E[1_{u \in K}] = \sum_{u \in U'} \alpha_u < |U'| \alpha^*. \label{firstm} \eeq
The second equality holds since $\alpha_u$ is the fraction of cliques containing $u$, and the inequality holds because every $u \in U'$ is a bad vertex and thus $\alpha_u < \alpha^*$ by definition. 

For any unordered pair of distinct vertices $u, v\in U'$, we want to count the fraction of cliques containing either $u$ or $v$. By the inclusion-exclusion formula, this is given by the fraction of cliques containing $u$, i.e., $
\alpha_u$, plus the fraction of cliques containing $v$, i.e., $\alpha_v$, minus the fraction of cliques containing both $u$ and $v$, which we denote by $p_{uv}$. In summary, the fraction of cliques containing either $u$ or $v$ is $\alpha_u + \alpha_v - p_{uv}$. We will use an averaging argument to give an upper bound on the average value of $\alpha_u + \alpha_v - p_{uv}$. First we lower bound the sum of $p_{uv}$ over the unordered pairs of distinct vertices $u,v \in U'$: 

\beq \sum_{u,v \in U'} p_{uv} = \sum_{u,v \in U'} E[1_{u,v \in K}] = E[\sum_{u,v\in U'} 1_{u,v \in K}] = E\left[ \binom{|K \cap U'|}{2}\right] \geq \binom{ E[ |K \cap U'|]}{2} = \binom{\sum_{u \in U'} \alpha_u}{2}. \label{cs}\eeq 
The inequality above is the Cauchy-Schwarz inequality while the last equality in (\ref{cs}) holds because of the last equality in (\ref{firstm}). 
Combining the results above, we have the sum over unordered pairs $u, v \in U'$ satisfies 
\begin{eqnarray*}
\sum_{u,v \in U'} \alpha_u + \alpha_v - p_{uv} &\leq& \sum_{u,v \in U'}  (\alpha_u + \alpha_v) - \binom{\sum_{u \in U'} \alpha_u}{2}\\
&=& \sum_{u,v \in U'}  (\alpha_u + \alpha_v) - \frac{\sum_{u \in U'} \alpha_u^2}{2} - \sum_{u,v \in U'} \alpha_u \alpha_v + \frac{\sum_{u \in U'} \alpha_u}{2} \\
&<&  \sum_{u,v \in U'}  (\alpha_u + \alpha_v - \alpha_u \alpha_v) + \frac{1}{2}(\alpha^*-\alpha^{*2})|U'|\\
&<& \sum_{u,v \in U'} (2\alpha^* - \alpha^{*2}) + 0.1034|U'|,
\end{eqnarray*}
where the first inequality holds by (\ref{cs}), the second line holds by expanding the binomial coefficient, the third line holds by rearranging terms and  $\alpha_u < \alpha^*<1/2$ for $u \in U'$ and $x-x^2$ is an increasing function for $x<1/2$, and the last inequality holds because $\alpha_u < \alpha^*<1$ for $u \in U'$ and thus $-(1-\alpha_u)(1-\alpha_v) < -(1-\alpha^*)(1-\alpha^*)$ which implies $\alpha_u + \alpha_v - \alpha_u \alpha_v < 2 \alpha^* - \alpha^{*2}$ . 
Averaging over all ${|U'| \choose 2}$ unordered pairs of distinct $u,v \in U'$, the average value of $\alpha_u + \alpha_v - p_{uv}$ is at most $$2\alpha^* - \alpha^{*2} + \frac{0.1034|U'|}{{|U'| \choose 2}}=2\alpha^* - \alpha^{*2} +\frac{0.2068}{|U'|-1}<1/2,$$
where in the last inequality we substituted in the value of $\alpha^*$ and used that $|U'| \geq 150$. Therefore, there exists a pair $u, v\in U'$ of distinct vertices such that the fraction of cliques containing either $u$ or $v$ is less than 1/2. 

We next show that contracting the edge $(u,v)$ to a vertex $u'$ yields a graph $G'$ with more cliques than $G$, contradicting that $G$ is a social graph. In $G$, the number of cliques not containing $u$ or $v$ is greater than the number of cliques containing $u$ or $v$ since the fraction of cliques containing $u$ or $v$ is less than 1/2. As the distance between $u$ and $v$ is more than two in the complement of $G$, then $u'$ is adjacent to all other vertices in $G'$. Thus, the number of cliques in $G'$ is twice the number of cliques in $G$ not containing $u$ or $v$; this is greater than the number of cliques in $G$, which equals to the number of cliques containing $u$ or $v$ plus the number of cliques not containing $u$ or $v$. This contradicts that $G$ is a social graph and completes the proof. 
\end{proof}

\begin{lem}\label{newlemmaneeded}
If $I$ is an independent set in a graph, then there is a vertex in $I$ which is in a fraction at most $1/(|I|+1)$ of the cliques in the graph. In particular, if the vertices in $I$ are good (i.e. are in a fraction at least $\alpha^*$ of the cliques), then $|I| \leq 2$. 
\end{lem}
\begin{proof}
As $I$ is an independent set, each clique contains at most one vertex in $I$. We can therefore partition the cliques into $|I|+1$ types, those that contain a particular vertex in $I$ and those that contain none of the vertices in $I$. For each clique containing a particular vertex in $I$, deleting that vertex yields a unique clique that contains none of the vertices in $I$. Thus, the cliques containing none of the vertices in $I$ is the most common type, and by averaging, one of the vertices in $I$ is in a fraction at most $1/(|I|+1)$ of the cliques in $G$. If $|I| \geq 3$, then there is a vertex in $I$ in a fraction at most $1/4<\alpha^*$ of the cliques, and such a vertex is therefore not good.  
\end{proof}

We use the previous two lemmas in the proof of the next lemma which shows that for a very dense social graph,  almost all vertices have degree in the complement at most one. 

\begin{lem}\label{nondeg}
In every social graph $G$ with maximum missing degree $\Delta$, all but at most $600\Delta^3$ vertices  
have degree at most one in the complement of $G$.
\end{lem}
\begin{proof}
Let $V'$ be the set of vertices in $G$ excluding those that are bad vertices or the non-neighbors of bad vertices. Therefore $|V'| \geq |V(G)| - 300\Delta^2(1 + \Delta) \geq |V(G)| - 600\Delta^3$ by Lemma \ref{contract}.

All the vertices in $V'$ and their non-neighbors in $G$ are good vertices. 
If every vertex $v \in V'$ has degree at most one in the complement of $G$, then we are done.
So suppose for the sake of contradiction that there exists a vertex $v \in V'$ whose set $T = \{u_1, \dots, u_d\}$ of non-neighbors has cardinality $d \geq 2$. As $v \in V'$, the set $T$ consists of only good vertices. The set $T$ forms a clique as otherwise there are nonadjacent $u_i,u_j \in T$, and $\{v,u_i,u_j\}$ form an independent set of order $3$ of good vertices, contradicting Lemma \ref{newlemmaneeded}.

For any subset $S\subset T$, let $k_S$ be the number of cliques $K$ in $G \setminus T$ whose set of common neighbors in $T$ is $S$. In other words, we are counting cliques $K \in G \setminus T$ such that $K \cup S$ is a clique and $K \cup S \cup \{u_i\}$ is not a clique for any $u_i \in T \setminus S$. 
Each of the $k_S$ cliques in $G \setminus T$ has $2^{|S|}$ ways of extending it to a clique in $G$. Therefore the total number of cliques in $G$ is 
$ \sum_{S \subset T} k_S 2^{|S|}$, where the summation is over all subsets of $T$ including the empty set. 

We also want to count the number of cliques containing a particular vertex $u_i \in T$. Notice that any clique $K$ in $G \setminus T$ can be extended to include $u_i$ if and only if the vertices of $K$ have common neighborhood in $T$ containing $u_i$. Therefore the number of cliques including $u_i$ is $\sum_{S \subset T, u_i \in S} k_S 2^{|S|-1}.$ Hence,  the fraction $\alpha_{u_i}$ of cliques in $G$ that contain $u_i$ satisfies
\[ \alpha_{u_i} = \frac{\sum_{S \subset T, u_i \in S} k_S 2^{|S|-1}}{\sum_{S \subset T} k_S 2^{|S|}}. \]

For any clique $K$ in $G \setminus T$, suppose its set of common neighbors in $T$ is $S$. If $S  = \emptyset$, then $K$ cannot be extended to add any vertex in $T$, and thus does not contribute to any clique in $G$ that contains a vertex in $T$. When $S \neq \emptyset$, then $K$ can be extended to $S \subset T$. This clique $K$ contributes to $2^{|S|}-1$ cliques in $G$ that contains a vertex in $T$ by extending $K$ to any non-empty subset of vertices in $S$. Therefore the fraction of cliques in $G$ that contains at least one vertex in $T$ is
\[ \frac{\sum_{S \subset T, S \neq \emptyset} k_S (2^{|S|}-1)}{\sum_{S \subset T} k_S 2^{|S|}}. \]

It will also be helpful to compute the number of cliques containing $v$. Notice that $T$ is the set of all the non-neighbors of $v$. As $v$ is not adjacent to any vertex in $T$, no clique in $G$ contains $v$ and a vertex in $T$. So there are only three types of cliques in $G$: 

\begin{enumerate}
\item Those containing $v$ and hence no vertex in $T$, 
\item those not containing $v$ nor any vertex in $T$, 
\item  those that contain at least one vertex in $T$ and hence does not contain $v$. 
\end{enumerate} 
The fraction of cliques of type $1$ is $\alpha_v$ by definition. The number of cliques of type 1 is equal to the number of cliques of type 2. This is because a clique $K$ of type 1 if and only if  $K \setminus \{v\}$ is a clique of type 2. Hence, the fraction of cliques of type 3, i.e., cliques containing at least one vertex in $T$, is $1-2\alpha_v$. Thus, we have
\[1-2\alpha_v = \frac{\sum_{S \subset T, S \neq \emptyset} k_S (2^{|S|}-1)}{\sum_{S \subset T} k_S 2^{|S|}},
\]
and therefore by rearranging the above equality we have
\beq 2 \alpha_v  = \frac{\sum_{S \subset T} k_S}{\sum_{S \subset T} k_S 2^{|S|}}. \label{alv} \eeq
Since $v$ and the $u_i$'s are all good vertices, we have the inequalities $\alpha_v \geq \alpha^*, \alpha_{u_i} \geq \alpha^*$ for all $1 \leq i \leq d$. We next show that these inequalities cannot simultaneously hold.

Consider a change of variables, letting $\beta_i = \frac{\sum_{|S| =i}2^i k_S}{\sum_{S \subset T} k_S 2^{|S|}}.$  Clearly $\sum_{i =0}^d \beta_i = 1$ and each $\beta_i \geq 0$. 
We have 
\begin{align*} \sum_{i = 1}^d \alpha_{u_i} &=\frac{ \sum_{i=1}^d  \sum_{S \subset T, u_i \in S} k_S 2^{|S|-1}}{\sum_{S \subset T} k_S 2^{|S|}}  =  \frac{  \sum_{S \subset T} \sum_{i=1}^d 1_{u_i \in S} k_S 2^{|S|-1}}{\sum_{S \subset T} k_S 2^{|S|}}  \\
&= \frac{  \sum_{S \subset T, S \neq \emptyset} k_S 2^{|S|-1}|S|}{\sum_{S \subset T} k_S 2^{|S|}} 
= \sum_{j=1}^d  \frac{1}{2}\frac{  \sum_{S \subset T, |S| =j} k_S 2^{|S|}|S|}{\sum_{S \subset T} k_S 2^{|S|}}.
 \end{align*}
By the definition of the $\beta_i$'s, we therefore have
\[ \sum_{i = 1}^d \alpha_{u_i}  = \frac{1}{2} \sum_{i = 1}^d i\beta_i. \]
Since $\alpha_{u_i} \geq \alpha^*$, multiplying both sides by $2/d$, we have 
\beq \sum_{i = 1}^d i\beta_i /d \geq 2 \alpha^*. \label{dalpha}
\eeq
We can also write $\alpha_v$, whose value is given in (\ref{alv}), in terms of the $\beta_i$'s. Clearly, 
\beq
2 \alpha_v  = \frac{\sum_{S \subset T} k_S}{\sum_{S \subset T} k_S 2^{|S|}} = \frac{\sum_{i=0}^d \sum_{|S| = i} k_S}{\sum_{S \subset T} k_S 2^{|S|}} = \sum_{i = 0}^d \beta_i/2^i. \label{alphav}
\eeq
From (\ref{dalpha}) and (\ref{alphav}), we have the following inequalities,
\beq\sum_{i = 1}^d i\beta_i /d \geq 2 \alpha^*, \ \ \sum_{i = 0}^d \beta_i/2^i \geq 2\alpha^*, \ \ \ \sum_{i=0}^d \beta_i = 1, \ \ \  ~\textrm{each}~\beta_i \geq 0. \label{system}\eeq
We want to show that the inequalities in (\ref{system}) above have no solution.
We view the above system of inequalities as the following linear programming problem (LP0):
 \begin{align}
 \text{(LP0)}  \ \ & \max \sum_{i = 1}^d i\beta_i /d,\label{ln1}\\
&\text{subject to: }\nonumber \\
 & \ \ \ \sum_{i = 0}^d \beta_i/2^i \geq 2\alpha^*, \label{ln2} \\
 & \  \ \ \sum_{i=0}^d \beta_i = 1, ~\textrm{each}~\beta_i \geq 0. \label{ln3}
\end{align}
We show that the objective function $\sum_{i = 1}^d i\beta_i /d$ in (\ref{ln1}) is always less than $2\alpha^*$, and thus obtain a contradiction to the first inequality in (\ref{system}).

 We first show that there cannot be three $\beta_i$'s which are positive in the optimal solution. We prove this by contradiction. Suppose the optimal solution has some $\beta_{i_1}, \beta_{i_2}, \beta_{i_3}$ being positive with $i_1 < i_2 < i_3$. We keep the rest of the $\beta_{i}$'s unchanged, and thus $\beta_{i_1} + \beta_{i_2} + \beta_{i_3} = a$ where $a = 1 - \sum_{i \neq i_1, i_2, i_3} \beta_i$ is fixed and independent from $\beta_{i_1}, \beta_{i_2}, \beta_{i_3}$. We use local adjustment to see this cannot be an optimal solution. From the equality we know that $\beta_{i_3} = a - (\beta_{i_1} + \beta_{i_2})$. Constraint (\ref{ln2}) states that $\beta_{i_1}/2^{i_1} + \beta_{i_2}/2^{i_2} + \beta_{i_3}/2^{i_3} \geq c$ where $c$ is a constant in terms of only the rest of the $\beta_i$'s, which we have fixed.  Writing $\beta_{i_3}$ in terms of $\beta_{i_1}, \beta_{i_2}$, we have
\[ \beta_{i_1}/2^{i_1} + \beta_{i_2}/2^{i_2} + \beta_{i_3}/2^{i_3} = \left( \frac{1}{2^{i_1}} - \frac{1}{2^{i_3}}\right) \beta_{i_1} + \left( \frac{1}{2^{i_2}} - \frac{1}{2^{i_3}}\right) \beta_{i_2} + \frac{1}{2^{i_3}} \left(\beta_{i_1} +\beta_{i_2}+\beta_{i_3}\right) \geq c, \] 
and thus 
\[
 \left( \frac{1}{2^{i_1}} - \frac{1}{2^{i_3}}\right) \beta_{i_1} + \left( \frac{1}{2^{i_2}} - \frac{1}{2^{i_3}}\right) \beta_{i_2} \geq c - a/2^{i_3}.
\]
Thus (LP0) is equivalent to maximize, after fixing the other $\beta_{i}$'s, the objective function $i_1 \beta_{i_1} + i_2 \beta_{i_2}  + i_3 \beta_{i_3} = i_3a - (i_3-i_1) \beta_{i_1} - (i_3 - i_2)\beta_{i_2}.$  
Therefore our new linear programming problem becomes the following linear programming problem (LP0')
\begin{align}
\text{(LP0')}  \ \ &  \max_{\beta_{i_1}, \beta_{i_2}} i_3a - (i_3-i_1) \beta_{i_1} - (i_3 - i_2)\beta_{i_2} \label{ln1'}\\
& \text{subject to:} \nonumber \\
  & \ \ \  \left( \frac{1}{2^{i_1}} - \frac{1}{2^{i_3}}\right) \beta_{i_1} + \left( \frac{1}{2^{i_2}} - \frac{1}{2^{i_3}}\right) \beta_{i_2} \geq c - a/2^{i_3}, \label{ln2'} \\
   & \ \ \  \beta_{i_1} + \beta_{i_2} \leq a, \label{ln3'} \\
  & \ \ \  \beta_{i_1}\geq 0, \label{ln4'} \\
  & \ \ \   \beta_{i_2} \geq 0.  \label{ln5'}
\end{align}
Clearly for this two-variable linear programming (LP0'), the four slopes are all different; the four constraints define a convex bounded polygon. We know the optimal solution should appear at an extreme point of the polygon. In other words, the optimal solution should be at the intersection of two constraints. If one of the two constraints is (\ref{ln4'}) or (\ref{ln5'}) then we are done, since in this case we have either $\beta_{i_1}$ or $\beta_{i_2}$ is $0$. Otherwise, one of the constraints that determines the optimal solution has to be (\ref{ln3'}). We are also done in this case since if we have equality in (\ref{ln3'}), then $\beta_{i_3} =a - (\beta_{i_1} + \beta_{i_2})=0$. In any case this linear program will give one of $\beta_{i_1}, \beta_{i_2}, \beta_{i_3}$ is 0. This contradicts the assumption that these three values are positive. Therefore the optimal solution to (\ref{ln1}) has at most two of the $\beta_i$'s are non-zero. 

Suppose by maximizing (\ref{ln1}) we have $\beta_h=0$ for all $h$ except for $h=i,j$. Without loss of generality we assume $0 \leq i < j$. By (\ref{ln3}) we have $\beta_i + \beta_j = 1$. By (\ref{ln2}), we obtain
$2 \alpha^* \leq \beta_i/2^i + \beta_j / 2^j < \beta_i/2^i + \beta_j / 2^{i} =\frac{1}{2^{i}}$. This implies $i=0$ as $2 \alpha^* >1/2$. Thus in (\ref{ln1}) we are just maximizing $j\beta_j /d$, and it needs to be at least $2\alpha^*$ by (\ref{system}). If $j=1$, we obtain $2\alpha^* \leq j\beta_j/d \leq  \beta_j/2$, implying $\beta_j \geq 4\alpha^*>1$, a contradiction. Hence, $j \geq 2$. Using (\ref{ln2}), (\ref{ln3}), the fact that $j \leq d$, and $2\alpha^* \leq j\beta_j/d \leq \beta_j$, we have 
$$2\alpha^* \leq \frac{\beta_i}{2^i}+\frac{\beta_j}{2^j} =\frac{\beta_0}{2^0}+\frac{\beta_j}{2^j} =( 1 - \beta_j) + \frac{\beta_j}{2^j} = 1 - \left(1-\frac{1}{2^j}\right)\beta_j \leq 1- \frac{3}{4} \beta_j \leq  1-\frac{3}{4} \cdot 2\alpha^*,$$
which would imply $\alpha^* \leq 2/7$, contradicting the choice of $\alpha^*$ (which is greater than $.29>2/7$), which completes the proof.  
\end{proof}

\subsection{On the number of cliques in very dense $H$-minor free graphs}\label{extremelydense} 
In this subsection, we prove a simple lemma that gives an essentially tight upper bound on the number of cliques an $H$-minor free graph can have in which the complement has maximum degree at most one. 


%

\begin{lem}\label{twocasesclique}
Let $H$ be a connected graph on $t$ vertices and $x$ be the size of a maximum matching in the complement of $H$. 
Let $G'$ be an $H$-minor free graph whose complement has maximum degree at most one. 

If $x < (2- 3/\log 3)t$, then the number of cliques in $G'$ is less than $3^{\frac{2}{3}t - \frac{x}{3}}$. If $x \geq (2- 3/\log 3)t$, then the number of cliques in $G'$ is less than $2^t$. 
\end{lem}
Tthe bounds in Lemma \ref{twocasesclique} are tight up to a factor $3$ given by considering a complement of a perfect matching or a clique of an appropriate order. The remarks after the proof of Lemma \ref{twocasesclique} provide more details. 

To prove Lemma \ref{twocasesclique}, we first characterize when $H$ can be a minor of $G'$. Notice that a graph $G'$ which has degree in the complement at most one consists of a complement of a perfect matching complete to a clique. So $G'$ has $2a+b$ vertices with the complement of the perfect matching on $2a$ vertices and the clique on $b$ vertices. 
\begin{lem}\label{notminor}
Let $H$ and $G'$ be the same as in Lemma \ref{twocasesclique}, where $G'$ has $2a+b$ vertices with the missing matching of $G'$ has $2a$ vertices and the clique has $b$ vertices. Then $G'$ is $H$-minor free if and only if 
\begin{align*}
&1.5a + b < t-0.5x \ ~~~\text{if} \ a \geq x,~\textrm{and}~\\
& 2a + b < t \ ~~~~~~~~~~~~~~\text{if} \ a < x .
\end{align*}
\end{lem}

\begin{proof}
The complement of $H$ has a maximum matching on $2x$ vertices. The remaining $t-2x$ vertices of $H$ is a clique as otherwise we can make an even bigger matching in the complement of $H$. 

Graph $H$ is a minor of $G'$ is equivalent to certain inequalities need to be satisfied.
When $a \geq x$, we can first embed the $x$ non-edges of the missing matching in $H$ into the $a$ missing edges of $G'$. Then $H$ is a minor of $G'$ is equivalent to embed a clique minor of order $t-2x$ into a graph on $2a+b-2x$ vertices which is a complement of a matching with size $a-x$. We have seen at the end of the proof of Lemma \ref{mainlem1} that the complement of a perfect matching of size $a-x$ has a maximum clique minor of size $\lfloor \frac{3}{2} (a-x)\rfloor$. Therefore, the Hadwiger number of $G'$ removing the missing matching on $2x$ vertices is at most $\lfloor \frac{3}{2}(a-x)\rfloor + b$. We need this number to be less than $t-2x$ as otherwise $H$ is a minor of $G'$. Thus,  in this case, $H$ is not a minor of $G'$ is equivalent to 
$t-2x > b+  \frac{3}{2}(a-x) $.
Rearranging the terms, we have
\beq
1.5a + b < t-0.5x. \nonumber
\eeq

Similarly, when $a <x$, we can first embed a missing matching on $2a$ vertices into the missing matching on $2a$ vertices in $G'$. Therefore, $H$ is a minor of $G'$ is equivalent to embed a missing matching with $x-a$ edges together with a clique on $t-2x$ vertices into a clique of size $b$. If $H$ is not a minor of $G'$, it means the total number of vertices is not enough, i.e., $2(x-a) + (t-2x) > b$. Rearranging, we have
\beq
2a + b < t.  \nonumber
\eeq
Thus we have obtained the necessary and sufficient conditions for $H$ not being a minor of $G'$ with the specified parameters $a, b$. 
\end{proof}

\begin{proof}[Proof of Lemma \ref{twocasesclique}]
Again, the number of vertices of $G'$ is $2a+b$ with the missing matching having $2a$ vertices and the clique having  $b$ vertices. 

The number of cliques in $G'$ is $3^a 2^b$ as for each non-adjacent pair of vertices in the missing matching, the two end vertices have $2^2-1=3$ possibilities of appearing in a clique since they cannot both appear. Thus the number of cliques in $G'$, by Lemma \ref{notminor}, is achieved by the following optimization problem (IP1) to the left. 

\begin{tabular}{p{6.5cm}p{0.1cm}p{6.5cm}}
{\begin{align}
& \text{(IP1)}   \ \  \max  \ 3^a 2^b = 2^{(\log 3) a  + b}  \nonumber \\ 
& \text{subject to:}  \nonumber \\
  &1.5a + b < t-0.5x  \ ~\text{if} \ a \geq x,~\textrm{and}~  \label{a>x}\\
 & 2a + b < t  ~~~~~~~~~~~~~~\text{if} \ a <x. \label{a<x} \\
 & a, b \geq 0, \ \ a, b \in \mathbb{Z}. \nonumber
\end{align}}
&
{\begin{align}
&  \nonumber \\
& \nonumber \\
& \xRightarrow[]{\text{\large{relaxation}}} \nonumber \\
& \nonumber
\end{align}}
&
 {\begin{align}
& \text{(LP1)}  \ \ \max  \ 3^a 2^b = 2^{(\log 3) a  + b}  \nonumber \\ 
& \text{subject to:}  \nonumber \\
&1.5a + b \leq t-0.5x  \ ~~~\text{if} \ a \geq x,~\textrm{and}~  \label{a>xrelax}\\
 & 2a + b \leq t  ~~~~~~~~~~~~~~~~\text{if} \ a < x. \label{a<xrelax} \\
 & a, b \geq 0, a,b \in \mathbb{R}. \nonumber
\end{align}}
\end{tabular}

Note that this is equivalent to maximizing $(\log 3)a+b$ under the same constraints. 
We relax (IP1) to another optimization problem (LP1)  above to the right by allowing $a, b$ to be real numbers and use the closure of the feasible region by replacing the strict inequalities (\ref{a>x}) and (\ref{a<x}) by inequalities (\ref{a>xrelax}) and (\ref{a<xrelax}). In other words, the optimal value of (IP1) is bounded above by the optimal value of the linear programming problem (LP1). 
When solving the linear system by using the constraint (\ref{a>xrelax}), we have that $b =0$ and $a \leq \frac{2t-x}{3}$, and thus $3^a 2^b \leq 3^{\frac{2}{3}t - \frac{x}{3}}.$ 
When solving the linear system by using the constraint (\ref{a<xrelax}) we have $a=0$ and $b \leq t$, and thus $3^a 2^b \leq 2^{t}.$ 
Therefore, by comparing the two bounds, the first bound is larger, i.e., $3^{\frac{2}{3}t - \frac{x}{3}} > 2^t$ when $x < (2- 3/\log 3)t$. The second bound is larger when otherwise. 
\end{proof}
\begin{rem}
We already saw that the optimal value of (IP1) is at most the optimal value of (LP1). We have already obtained the optimal value of (LP1). We show  that the optimal value of (IP1) is at least $\frac{1}{3}$ of the optimal value for (LP1). If the optimal solution of (LP1) is $(a^*, b^*)$, we have seen that $a^* = 0$ or $b^* = 0$. If $a^* = 0$, then the lattice point $(0, \lceil b^* \rceil - 1)$ is in the feasible region of (IP1) and the value of the objective functions in (IP1) for this point is $3^{a^*} 2^{\lceil b^* \rceil - 1} \geq \frac{1}{2} \left(3^{a^*}2^{b^*}\right)$, meaning that the value of (IP1) is at least half the optimal value of (LP1). Similarly, if $b^* = 0$, then the lattice point $(\lceil a^* \rceil - 1, 0)$ is in the feasible region of (IP1) and $ 3^{\lceil a^* \rceil - 1} 2^{b^*}\geq \frac{1}{3} \left(3^{a^*}2^{b^*}\right)$. Therefore the optimal value of (IP1) is within a factor $3$ of the optimal value of (LP1). This means that the bound in Lemma \ref{twocasesclique} is sharp up to a factor 3. 
\end{rem}

\subsection{Proof of Theorem \ref{generalmain}}
The proof of Theorem \ref{generalmain} has some similarities to the proof of the special case for cliques, Theorem \ref{main}, given in the previous section. We use the peeling process described in Subsection \ref{proofofmain} to peel off vertices of small degree, and show that they contribute only a small factor to the number of cliques. We end up with a dense graph. From the results obtained in the previous subsection, we understand how to maximize the number of cliques in the dense setting under the given constraints by Lemma \ref{twocasesclique}.

\begin{proof}[Proof of Theorem \ref{generalmain}]
We begin by bounding the number of cliques by the peeling process described in Subsection \ref{proofofmain}. Let $n_i$ denote the number of vertices in $G_i$. For a clique $K$ on $s$ vertices, let $r=r(K)$ be the least positive integer such that $n_r \leq 0.99t$ or $n_{r+1} \geq n_{r}-n_{r}^{1/4}$ or $r=s$. 

We first give a bound on $r$. Since $H$ has $t$ vertices and $G$ is $H$-minor free, it is also $K_t$-minor free. The result of Thomason \cite{Th1} implies, as $G$ is $K_t$-minor free, that every subgraph of it has a vertex of degree at most $d:=t\sqrt{\log t}$ (here we assume $t$ is sufficiently large). Hence $n_1 \leq d+1$. We have $n_i < n_{i-1}-n_{i-1}^{1/4}$ for each $2 \leq i \leq r$. In particular, if $j \leq r$ and $j-i \geq (n_i/2)^{3/4}$, then $n_j<n_i/2$. Indeed, otherwise we have $n_j<n_i-(j-i)(n_i/2)^{3/4} \leq n_i/2$, a contradiction. It follows that $$r \leq 1+\sum_{h \geq 0} 1+ ((d+1)/2^{h+1})^{3/4} \leq 2(t\sqrt{\log t})^{3/4}:= 2t^{3/4} (\log t)^{3/8}=r_0.$$

We next give a bound on the number of choices for $v_1,\ldots,v_r$. We have $n_0=n$ choices for $v_1$. We use the weak estimate that the number of choices for $v_2,\ldots,v_r$ having picked $v_1$ is at most 
${|G_1|-1 \choose \leq r_0} \leq r_0 \binom{t\sqrt{\log t}}{r_0} \leq r_0 \left(\frac{e t \sqrt{\log t}}{r_0}\right)^{r_0} \leq  r_0 t^{r_0} \left(   \frac{e\sqrt{\log t}}{r_0}\right)^{r_0}  \leq 2^{3t^{3/4}(\log t)^{11/8}}.$
We thus have at most $n2^{3t^{3/4}(\log t)^{11/8}}$ choices for $v_1,\ldots,v_r$. 

Recall that $G$ is an $H$-minor free graph on $n$ vertices and our goal is to bound the number of cliques in $G$. We have already bounded the number of choices for the first $r$ vertices, and it suffices to bound the number of choices for the remaining vertices. We split the cliques into three types: those with $n_r \leq 0.99t$, those with $r=s$ and $n_r > 0.99t$, and those with $r<s$, $n_r>0.99t$, and $n_{r+1} \geq n_r-n_r^{1/4}$. 

We first bound the number of cliques with $n_r \leq 0.99t$. As there are at most $0.99t$ possible remaining vertices to include after picking $v_1,\ldots,v_r$ for the clique, then there are at most $2^{0.99t}$ ways to extend these vertices to a  clique. We thus get at most $n2^{0.99t+3t^{3/4}(\log t)^{11/8}}$ cliques of the first type. 

We next bound the number of cliques with $r=s$. We saw that this is at most $n2^{3t^{3/4}(\log t)^{11/8}}$.

Finally, we bound the number of cliques with  $n_r \geq 0.99t$, $r<s$, and $n_{r+1} \geq n_r-n_r^{1/4}$. In this case, in $G_r$, $v_r$ has the minimum degree, and it and its non-neighbors are not in $G_{r+1}$, which has $n_{r+1}$ vertices. Thus, the maximum degree $\Delta$ of the complement of $G_r$ satisfies $\Delta < n_{r+1}-n_r \leq n_r^{1/4} \leq (t\sqrt{\log t})^{1/4} = t^{1/4}(\log t)^{1/8}$. 

Let $\hat G$ be a minor of $G_r$ which has the most cliques, so $\hat G$ is also $H$-minor free and is a social graph. By the definition of $\hat G$, we have $c(\hat G) \geq c(G_r)$. It is clear that the maximum missing degree in $\hat G$ is at most that of $G_r$, which is $\Delta$. Note that we may assume $\hat G$ is obtained from $G_r$ only by edge contractions (no edge or vertex deletions) as avoiding the deletions cannot decrease the number of cliques. By Lemma \ref{nondeg}, removing at most $600\Delta^3$ vertices in $\hat G$, the remaining induced subgraph $G'$ of $\hat G$ has degree in the complement at most one. Lemma \ref{twocasesclique} tells us that the total number of cliques in $G'$ is at most 
$3^{\frac{2}{3}t - \frac{x}{3}}$ when $x < (2- 3/\log 3)t$; and $2^t$ when $x \geq (2- 3/\log 3)t$. There are at most $600\Delta^3$ vertices in $\hat G \setminus G'$. Thus in $\hat G$, the number of cliques is at most $2^{600\Delta^3}$ times the number of cliques in $G'$. Therefore, in $\hat G$, there are at most $3^{\frac{2}{3}t - x/3 + 600\Delta^3} \leq 3^{2t/3 - x/3+ 600t^{3/4}(\log t)^{3/8}}$ cliques when $x < (2- 3/\log 3)t$; and at most $2^{t+600\Delta^3}=2^{t + 600t^{3/4}(\log t)^{3/8}}$ cliques when $x \geq (2- 3/\log 3)t$. Recall that the $c(\hat G)$ is no less than $c(G_r)$. 
We thus obtained an upper bound on the number of ways of completing a clique in $G$ having picked the first $r$ vertices. We thus get at most $n2^{r_0}3^{2t/3- x/3+600t^{3/4}(\log t)^{3/8}}=n3^{2t/3 - x/3 + t^{3/4+o(1)}}$ cliques of the last type if $x < (2- 3/\log 3)t$, and at most $n2^{r_0}2^{t+600t^{3/4}(\log t)^{3/8}}=n2^{t+t^{3/4+o(1)}}$ cliques of the last type if otherwise. 

Adding up all possible cliques, we get at most $3^{2t/3- x/3+t^{3/4 + o(1)}}n$ cliques if $x < (2- 3/\log 3)t$, and at most $2^{t + t^{3/4+o(1)}} n$ cliques if $x \geq (2- 3/\log 3)t$. In the first case, the bound is tight up to the error term by considering the disjoint union of perfect matchings of size just less than $2t/3-x/3$. In the second case, the bound is tight by considering up to the error term by considering a disjoint union of cliques of order $t-1$. 
\end{proof}

We next show how to use the previously obtained results in order to establish the maximum number of cliques a $K_t$-minor free graph on $n$ vertices can have when $t \leq n < 4n/3$. 

\begin{thm}
\label{main7} 
Every $K_t$-minor free graph on $n$ vertices with $t \leq n < 4t/3$ has at most $2^{4t-3n+2(n-t)\log_2 3 +o(t)}$ cliques.  
\end{thm}
\begin{proof}
We apply the peeling process until we arrive at an induced subgraph with maximum missing degree $\Delta$ at most $t^{1/4}$. In each step, we remove a vertex with all all its non-neighbors, whose cardinality is at least $t^{1/4}$ before the peeling process terminates. Thus there are at most $n/t^{1/4}$ vertices in a clique before arriving at the dense induced subgraph. Hence, we have at most $${n \choose \leq n/t^{1/4}} \leq (et^{1/4})^{n/t^{1/4}}=2^{t^{3/4+o(1)}}$$
choices for the vertices of a clique before we arrive at the dense induced subgraph with maximum missing degree $\Delta \leq t^{1/4}$. By Lemma \ref{nondeg}, all but at most $600\Delta^3 \leq 600t^{3/4}$ vertices have missing degree at most one. These at most $600t^{3/4}$ vertices contribute a factor at most $2^{600t^{3/4}} \leq 2^{t^{3/4+o(1)}}$ to the number of cliques. The remaining induced subgraph has $n'<n$ vertices, no $K_t$-minor, and has maximum missing degree at most one. Let $x'$ be the size of the maximum missing matching. So the number of cliques in this induced subgraph is $3^{x'}2^{n'-2x'}$. This is maximized given $n'$ if $x'$ is as small as possible so that it is still $K_t$-minor free, which is when $x'=2(n'-t)+1$. The number of cliques in this is case is $2^{4t-3n'-2+(2n'-2t+1)\log_2 3}$. This is an increasing function of $n'$, which is at most $n$, and thus we get an upper bound when we substitute $n$ for $n'$. With the factors we get from the peeling process and the few vertices in the remaining induced subgraph that have missing degree more than one, in total we get at most $2^{4t-3n+2(n-t)\log_2 3+o(t)}$ cliques, completing the proof.  
\end{proof}

\section{Counting cliques in a graph in a minor-closed family}\label{secfamily}
In this section we generalize the result in Theorem \ref{generalmain} further to prove Theorem \ref{family} bounding the number of cliques in a graph on $n$ vertices that belongs to a minor-closed family of graphs that is closed under disjoint union. As in the previous sections, we determine the exponential constant.  

\begin{proof}[Proof of Theorem \ref{family}]
Similar to the previous section, we first apply the peeling process with the same parameters to reduce to a dense graph $G_r$ whose maximum missing degree $\Delta$ is at most $t_1^{1/4+o(1)}$. Let $\hat G$ be the minor of $G_r$ with the most cliques. By  Lemma \ref{nondeg}, we know that in $\hat G$ after removing at most $600\Delta^3$ vertices, the remaining induced subgraph of $\hat G$, denoted by $G'$, has degree in the complement at most one. Since $\mathcal{G}$ is minor-closed, then $G'$ is also in $\mathcal{G}$, i.e., it forbids $H_1,\ldots,H_{\ell}$ as minors.

To bound the number of cliques in $G'$, we use the same argument as in Lemma \ref{twocasesclique}. Suppose $G'$ is the complement of a matching on $2a+b$ vertices consisting of the complement of a perfect matching on $2a$ vertices which is complete to a clique on $b$ vertices, so $G'$ has $3^a2^b$ cliques. We want that $H_i$ is not a minor in $G'$ for each $1 \leq i \leq \ell$. Lemma \ref{notminor} tells us we need, for each $1 \leq i \leq \ell$, that 
$1.5a + b < t_i-0.5x_i$ if $a \geq x_i$ and 
$2a + b < t_i$ if $ a \leq x_i$. 
Therefore when bounding the maximum number of cliques in $G'$, similar to what we did in Lemma \ref{twocasesclique}, we want to solve the following optimization problem (IP2) below on the left. We relax it to a related linear programming problem (LP2) on the right below. Note that this is indeed a linear program as it is equivalent to maximizing $(\log_2 3) a  + b$ under the same conditions.

\begin{tabular}{p{6.8cm}p{0.1cm}p{6.8cm}}
{\begin{align}
& \text{(IP2)}  \ \  \max  \ 3^a 2^b = 2^{(\log 3) a  + b}  \label{obji}\\ 
& \text{subject to:} \nonumber\\
& \text{For each}\ 1 \leq i \leq \ell: \nonumber  \\
 & \ \ 1.5a + b < t_i-0.5x_i  \ \text{if} \ a \geq x_i,~\textrm{and}~  \label{con1i}\\
 & \  \ 2a + b < t_i  ~~~~~~~~~~~~~\text{if} \ a <x_i. \label{con2i} \\
 &  a, b \geq 0, \ \ a, b \in \mathbb{Z}. \nonumber
\end{align}}
&
{\begin{align}
&  \nonumber \\
& \nonumber \\
& \xRightarrow[]{\text{\large{relaxation}}} \nonumber \\
& \nonumber
\end{align}}
&
{\begin{align}
& \text{(LP2)}  \ \  \max  \ 3^a 2^b = 2^{(\log 3) a  + b}  \label{obj}\\ 
& \text{subject to:} \nonumber\\
& \text{For each}\ 1 \leq i \leq \ell: \nonumber  \\
 & \ \ 1.5a + b \leq t_i-0.5x_i  \ \text{if} \ a \geq x_i,~\textrm{and}~  \label{con1}\\
 & \  \ 2a + b \leq t_i  ~~~~~~~~~~~~~\text{if} \ a < x_i. \label{con2} \\
 &  a, b \geq 0, \ \ a, b \in \mathbb{R}. \nonumber
\end{align}}
\end{tabular}


The optimal value in (IP2) is bounded above by the optimal value in (LP2). 
To solve the linear program (LP2), it is helpful to visually work out each constraint. The inequality given in (\ref{con2}) (for each $i$) is the region given by $b \leq t_i - 2a$ when $0 \leq a < x_i$. The closure of this region is on or below the blue line segment $AF$ in the closed first quadrant in Figure \ref{convi}. 
The constraint (\ref{con1}) is defined by $a \geq x_i$ and the line $b \leq (t_i - 0.5x_i) - 1.5a$ in the closed first quadrant. This region is on or below the green line segment $FD$ in the closed first quadrant. 
The slope of the green line $b \leq (t_i - 0.5x_i) - 1.5a$ has magnitude smaller than the blue line. 
Therefore the feasible region of (LP2) for each $i$ can be visualized as the shaded region in Figure \ref{convi}, the region in the closed first quadrant below the piecewise linear curve (with one bend) with segments $AF$ and $FD$.

\begin{figure}[h]
\centering
\definecolor{uuuuuu}{rgb}{0.26666666666666666,0.26666666666666666,0.26666666666666666}
\definecolor{qqzzqq}{rgb}{0.,0.6,0.}
\definecolor{qqqqff}{rgb}{0.,0.,1.}
\definecolor{ttqqqq}{rgb}{0.2,0.,0.}
\begin{tikzpicture}[line cap=round,line join=round,>=triangle 45,x=2cm,y=1.2cm]
\draw[->,color=black] (-0.2,0.) -- (3,0.);
\foreach \x in {,1.,2.,3.}
\draw[->,color=black] (0.,-0.2) -- (0.,5.5);
\foreach \y in {,2.,4..}
\draw[color=black] (0pt,-8pt) node[right] {\footnotesize $0$};
\clip(-1,-0.5) rectangle (10,5.5);
\fill[color=ttqqqq,fill=ttqqqq,pattern=north east lines,pattern color=ttqqqq] (0.,5.) -- (0.,0.) -- (2.6666666666666665,0.) -- (2.,1.) -- cycle;
\draw [line width=1.2pt,color=qqqqff] (0.,5.)-- (2.5,0.);
\draw [line width=1.2pt,color=qqzzqq] (0.,4.)-- (2.6666666666666665,0.);
\draw [line width=1.2pt,color=qqqqff] (2.67958,5.11)-- (2.29525,5.11);
\draw [line width=1.2pt,color=qqzzqq] (2.67958,4.67)-- (2.29525,4.67);
\draw (2.7,5.3) node[anchor=north west] {\footnotesize{Line $b \leq t_i - 2a$}};
\draw (2.7,4.9) node[anchor=north west] {\footnotesize{Line $b \leq (t_i - 0.5x_i) - 1.5a$}};
\draw (3,0.053536600751215524) node[anchor=north west] {$a$};
\draw (-0.2,5.5) node[anchor=north west] {$b$};
\draw (-0.2,4.281229985490518) node[anchor=north west] {\footnotesize$C: $};
\draw (-0.92,4.1) node[anchor=north west] {\footnotesize$(0,t_i - 0.5x_i)$};
\draw (1.8,-0.07) node[anchor=north west] {\footnotesize$x_i$};
\draw (2.2,-0.03) node[anchor=north west] {\footnotesize$B:(\frac{t_i}{2},0)$};
\draw (2.,2.65848)-- (2.,-0.1);
\begin{scriptsize}
\draw [fill=ttqqqq] (0.,5.) circle (1.0pt);
\draw[color=ttqqqq] (0.32,5) node {$A:(0, t_i)$};
\draw [fill=ttqqqq] (2.5,0.) circle (1.0pt);
\draw [fill=ttqqqq] (0.,4.) circle (1.0pt);
\draw [fill=ttqqqq] (2.6666666666666665,0.) circle (1.0pt);
\draw[color=ttqqqq] (3.1,0.2) node {$D:(\frac{t_i - 0.5x_i}{1.5}, 0)$};
\draw [fill=uuuuuu] (0.,0.) circle (1.5pt);
\draw [fill=uuuuuu] (2.,1.) circle (1.5pt);
\draw[color=uuuuuu] (2.6,1.) node {$F:(x_i, t_i - 2x_i)$};
\end{scriptsize}
\end{tikzpicture}
\caption{Constraints (\ref{con1}) and (\ref{con2}) visualization.  \footnotesize{The blue line $b \leq t_i - 2a$ coming from (\ref{con2}) intersects the axis $a = 0$ at $A: (0, t_i)$ and  intersects the axis $b = 0$ at $B: (t_i/2, 0)$. The green line $b \leq (t_i - 0.5x_i) - 1.5a$ coming from (\ref{con1}) intersects the axis $a = 0$ at $C: (0, t_i - 0.5x_i)$ and intersects the axis $b = 0$ at point $D: (\frac{t_i - 0.5x_i}{1.5}, 0)$. $F$ is the intersection of the two lines and it can be easily solved that $F = (x_i, t_i - 2x_i)$. }} \label{convi}
\end{figure}
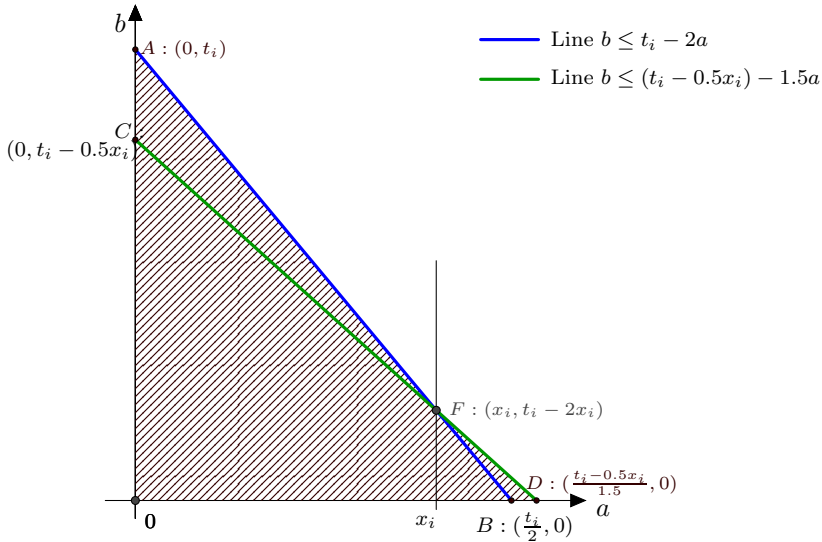

\begin{figure}[H]
\centering
\definecolor{ttqqqq}{rgb}{0.2,0.,0.}
\begin{tikzpicture}[line cap=round,line join=round,>=triangle 45,x=1.0cm,y=0.6cm]
\draw[->,color=black] (-0.2,0.) -- (4,0.);
\foreach \x in {-1.,1.,2.,3.,4.,5.,6.,7.}
\draw[->,color=black] (0.,-0.2) -- (0.,7);
\foreach \y in {,2.,4.,6.,8.}
\draw[color=black] (0pt,-7pt) node[right] {\footnotesize $0$};
\clip(-0.3,-0.4) rectangle (7,7);
\fill[color=ttqqqq,fill=ttqqqq,pattern=north east lines,pattern color=ttqqqq] (0.,6.) -- (0.5657349737952752,3.5552624345756887) -- (1.2489985511268813,2.8292948836608587) -- (1.4198144454597827,1.9111594516215151) -- (2.6262016991859,0.8222081252492706) -- (2.7329616331439635,0.3097604422505672) -- (3.341493256704925,0.) -- (0.,0.) -- cycle;
\draw (0.,6.)-- (0.5657349737952752,3.5552624345756887);
\draw (0.5657349737952752,3.5552624345756887)-- (1.2489985511268813,2.8292948836608587);
\draw (1.2489985511268813,2.8292948836608587)-- (1.4198144454597827,1.9111594516215151);
\draw (1.4198144454597827,1.9111594516215151)-- (2.6262016991859,0.8222081252492706);
\draw (2.6262016991859,0.8222081252492706)-- (2.7329616331439635,0.3097604422505672);
\draw [color=ttqqqq] (0.,6.)-- (0.5657349737952752,3.5552624345756887);
\draw [color=ttqqqq] (0.5657349737952752,3.5552624345756887)-- (1.2489985511268813,2.8292948836608587);
\draw [color=ttqqqq] (1.2489985511268813,2.8292948836608587)-- (1.4198144454597827,1.9111594516215151);
\draw [color=ttqqqq] (1.4198144454597827,1.9111594516215151)-- (2.6262016991859,0.8222081252492706);
\draw [color=ttqqqq] (2.6262016991859,0.8222081252492706)-- (2.7329616331439635,0.3097604422505672);
\draw [color=ttqqqq] (2.7329616331439635,0.3097604422505672)-- (3.341493256704925,0.);
\draw [color=ttqqqq] (3.341493256704925,0.)-- (0.,0.);
\draw [color=ttqqqq] (0.,0.)-- (0.,6.);
\begin{scriptsize}
\draw [fill=ttqqqq] (0.,6.) circle (2.5pt);
\draw[color=ttqqqq] (0.07463927758818326,6.373724691068557) node {$A$};
\draw[color=ttqqqq] (-0.1,6.8) node[left] {$b$};
\draw [fill=ttqqqq] (0.5657349737952752,3.5552624345756887) circle (2.5pt);
\draw[color=ttqqqq] (0.6404669275659196,3.9395981968247162) node {$B$};
\draw [fill=ttqqqq] (1.2489985511268813,2.8292948836608587) circle (2.5pt);
\draw[color=ttqqqq] (1.3237305048975256,3.2136306459098862) node {$C$};
\draw [fill=ttqqqq] (1.4198144454597827,1.9111594516215151) circle (2.5pt);
\draw[color=ttqqqq] (1.4945463992304273,2.295495213870543) node {$D$};
\draw [fill=ttqqqq] (2.6262016991859,0.8222081252492706) circle (2.5pt);
\draw[color=ttqqqq] (2.7009336529565444,1.2065438874982979) node {$E$};
\draw [fill=ttqqqq] (2.7329616331439635,0.3097604422505672) circle (2.5pt);
\draw[color=ttqqqq] (2.807693586914608,0.6940962044995947) node {$F$};
\draw [fill=ttqqqq] (3.341493256704925,0.) circle (2.5pt);
\draw[color=ttqqqq] (3.4162252104755697,0.37381640262540516) node {$G$};
\draw[] (4,0) node[below] {$a$};
\draw [fill=ttqqqq] (0.,0.) circle (1.5pt);
\end{scriptsize}
\end{tikzpicture}
\caption{Illustration of the feasible region (polygon) of (LP2). \footnotesize{The lower envelope $\mathcal{P}$ is the union of the closed line segments $AB, BC, CD, DE, EF,FG$.} }\label{pol}
\end{figure}
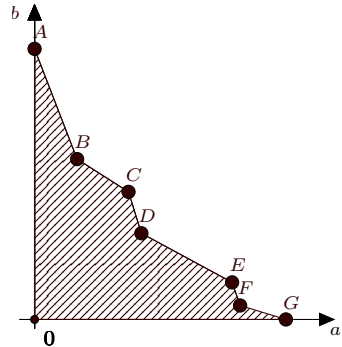

After intersecting all the shaded areas bounded by the one-bended piecewise linear curves for each graph $H_i$, we obtain the feasible region of (LP2), which is a polygon of the type as in Figure \ref{pol}. There are only two slopes, $-1.5$ and $-2$, in the non-axis line segments bounding this polygon. We call the closure of this set of one-bended line segments which bounds the resulting polygon the \emph{lower envelope $\mathcal{P}$}. For example, in Figure \ref{pol}, the lower envelope is the union of the closed line segments $AB, BC, CD, DE, EF,FG$. Note that not all the bended line segments for each $H_i$ should participate in the lower envelope: a bended line segment can be always strictly higher than the lower envelope and thus will not be in the intersection of all the shaded areas for $H_i$'s.  Since our objective function has slope in between the two slopes appearing in the lower envelope (as $1.5 < \log_2 3 < 2$), we know that the optimal solution should appear only at an extreme point of the lower envelope. For example, in Figure \ref{pol}, the extreme points of the lower envelope $\mathcal{P}$ are $A,B,C,D,E,F,G$. In contrast to the case $\ell=1$ considered in the previous section, the optimal solution can appear in an extreme point which is not on the axises (so neither $a = 0$ nor $b=0$), and we will show an example soon. In summary, the solution of (LP2) is
\beq 
\max_{(a_j, b_j)~\text{an extreme point of the lower envelope $\mathcal{P}$}} 2^{(\log_2 3)a_j +b_j}.     \label{linearsys}
\eeq

To see that the optimal solution of (IP2), which gives us the maximum number of cliques in $G'$, is within a factor $6$ of the optimal solution of (LP2), let the point $p = (a^*, b^*)$ be the optimal solution of (LP2). The feasible region for (IP2) are exactly the lattice points of the feasible region of (LP2) excluding the lattice points on the lower envelope $\mathcal{P}$. Therefore the closest lattice point $p' < p$ but $p' \neq p$ is in the feasible region of (IP2). Since both coordinates of $p'$ differ from $p$ by at most 1, by plugging in $p'$ into the objective function in (IP2), it gives a value at least $\frac{1}{6}$ the optimal value for (LP2). Since the optimal value for (IP2) is at least the value by plugging in $p'$, we have that the optimal value for (IP2) is within a factor $6$ of the optimal value for (LP2). Thus the bound we achieved by solving (LP2) instead is tight up to a factor $6$. 

We now bound the number of cliques a graph $G \in \mathcal{G}$ on $n$ vertices can have. The argument is almost the same as in the proof of Theorem \ref{generalmain}. Thus we omit the repetitive details. The number of cliques in $G'$ we have shown is at most (\ref{linearsys}), and the same argument as in Theorem \ref{generalmain} shows that the number of cliques in the graph $G$ is at most $2^{t_1^{3/4+o(1)}}n$ times this bound.

Tightness of the bound comes from taking a disjoint union of copies of a graph $G'$ on $2a'+b'$ vertices which is the complement of a matching of size $a'$, where $p'=(a',b')$ is the lattice point in the feasible region defined earlier. 
\end{proof}

When there is only one connected graph forbidden as a minor, as we have seen in Section \ref{section3}, the optimal solution for the relaxed system (LP1) satisfies either $b=0$ or $a=0$. To relate to (IP1), this means that an optimal construction of a graph forbidding $H$ as minor and with the maximum number of cliques is close to being a complement of a perfect matching or a clique. However, for the general case, when there are multiple forbidden minors, the optimal construction can be far from either. 
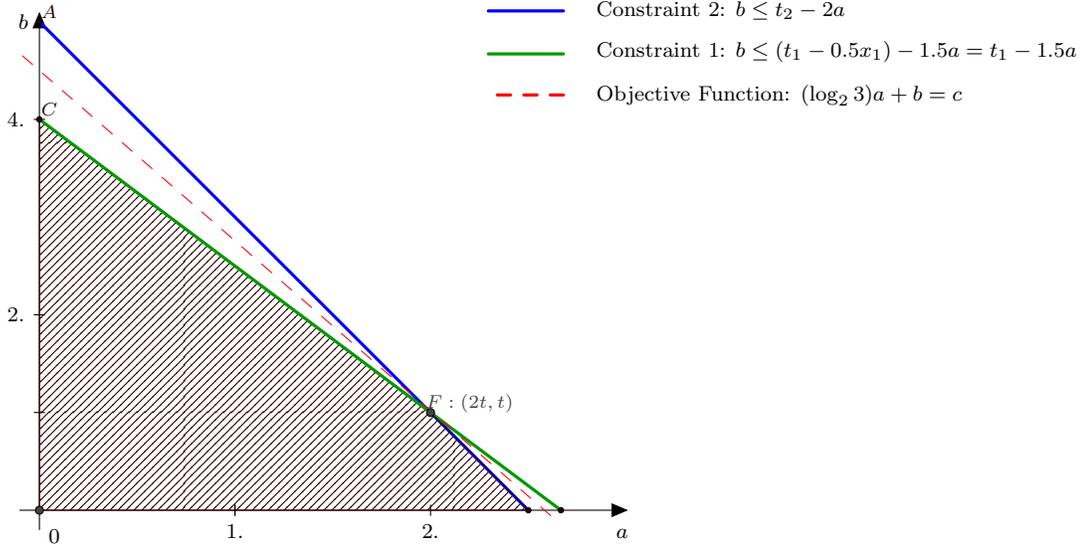
\begin{figure}[H]
\centering
\definecolor{ffqqqq}{rgb}{1.,0.,0.}
\definecolor{uuuuuu}{rgb}{0.26666666666666666,0.26666666666666666,0.26666666666666666}
\definecolor{qqzzqq}{rgb}{0.,0.6,0.}
\definecolor{qqqqff}{rgb}{0.,0.,1.}
\definecolor{ttqqqq}{rgb}{0.2,0.,0.}
\begin{tikzpicture}[line cap=round,line join=round,>=triangle 45,x=2.6cm,y=1.3cm]
\draw[->,color=black] (-0.1,0.) -- (3.01,0.);
\foreach \x in {1.,2.,}
\draw[shift={(\x,0)},color=black] (0pt,2pt) -- (0pt,-2pt) node[below] {\footnotesize $\x$};
\draw[->,color=black] (0.,-0.2) -- (0.,5.1);
\foreach \y in {,2.,4.}
\draw[shift={(0,\y)},color=black] (2pt,0pt) -- (-2pt,0pt) node[left] {\footnotesize $\y$};
\draw[color=black] (0pt,-10pt) node[right] {\footnotesize $0$};
\draw[color=black] (2.98, -0.1) node[below] {\footnotesize $a$};
\draw[color=black] (-0.01, 5.) node[left] {\footnotesize $b$};
\clip(-0.8,-0.2) rectangle (10,6);
\fill[color=ttqqqq,fill=ttqqqq,pattern=north east lines,pattern color=ttqqqq] (0.,4.) -- (0.,0.) -- (2.5,0.) -- (2.,1.) -- cycle;
\draw [line width=1.2pt,color=qqqqff] (0.,5.)-- (2.5,0.);
\draw [line width=1.2pt,color=qqzzqq] (0.,4.)-- (2.6666666666666665,0.);
\draw [color=ttqqqq] (0.,4.)-- (0.,0.);
\draw [color=ttqqqq] (0.,0.)-- (2.5,0.);
\draw [color=ttqqqq] (2.5,0.)-- (2.,1.);
\draw [color=qqzzqq] (2.,1.)-- (0.,4.);
\draw [dash pattern=on 5pt off 5pt,color=ffqqqq] (-0.086242997434229,4.65)-- (2.64,-0.1);
\draw [line width=1.2pt,color=qqqqff] (2.67958,5.11)-- (2.29525,5.11);
\draw [line width=1.2pt,color=qqzzqq] (2.67958,4.67)-- (2.29525,4.67);
\draw [line width=1.2pt,dash pattern=on 5pt off 5pt,color=ffqqqq] (2.67958,4.25)-- (2.29525,4.25);
\draw (2.8,5.3) node[anchor=north west] {\footnotesize{Constraint 2: $b \leq t_2 - 2a$}};
\draw (2.8,4.9) node[anchor=north west] {\footnotesize{Constraint 1: $b \leq (t_1 - 0.5x_1) - 1.5a = t_1 - 1.5a$}};
\draw (2.8,4.45) node[anchor=north west] {\footnotesize{Objective Function: $(\log_2 3 )a + b = c$} };
\begin{scriptsize}
\draw [fill=ttqqqq] (0.,5.) circle (1.0pt);
\draw[color=ttqqqq] (0.05,5.1) node {$A$};
\draw [fill=ttqqqq] (2.5,0.) circle (1.0pt);
\draw [fill=ttqqqq] (0.,4.) circle (1.0pt);
\draw[color=ttqqqq] (0.05,4.1) node {$C$};
\draw [fill=ttqqqq] (2.6666666666666665,0.) circle (1.0pt);
\draw [fill=uuuuuu] (0.,0.) circle (1.5pt);
\draw [fill=uuuuuu] (2.,1.) circle (1.5pt);
\draw[color=uuuuuu] (2.2,1.1) node {$F: (2t, t)$};
\end{scriptsize}
\end{tikzpicture}
\caption{Example of Forbidding $H_1, H_2$. \footnotesize{Constraint $i$ is the feasible region coming from $H_i$. The intersection of the two feasible regions give the feasible region for (LP2). The objective function is optimized when the appropriate $c$ is chosen such that the red dashed line goes through $F$, the intersection of the blue and green lines.}} \label{example}
\end{figure}
We will consider and example with $\ell=2$ forbidden minors. Let $H_1$ be a clique on $t_1 = 4t$ vertices, so $x_1 = 0$. Let $H_2$ be a graph on $t_2 = 5t$ vertices with $x_2 =  t_2/2 = 2.5t$, i.e., $H_2$ is a complement of a perfect matching on $5t$ vertices. We next consider the integer program (IP2) and the linear program (LP2) in this case. The objective function (\ref{obj}) translates into maximizing $(\log_2 3)a +b$, which is of some constant value $c$ on each line parallel to the red line in Figure \ref{example}. For $H_1$, the constraint (\ref{con1}) is: when $a \geq x_1 =0$, we need $1.5a + b < t_1  - 0.5x_1 = 4t$, which is equivalent to $b < 4t - 1.5a$. Constraint (\ref{con2}) can be ignored here since (\ref{con2}) requires $a < x_1 = 0$. Therefore for $H_1$ we only have a straight line segment $b < 4t - 1.5a$ for $a, b\geq 0$. This is shown as the green line in Figure \ref{example}. For $H_2$, the constraint (\ref{con2}) which is for when $a < x_1 =5t$ is $2a + b < t_2 = 5t$, which is equivalent to $b < 5t-2a$. This line intersects $b=0$ line at $(2.5t, 0)$. This point is also the starting point for the line segment shown in (\ref{con1}). Therefore the constraint for $H_2$ is again a straight line segment as shown as the blue line in Figure \ref{example}. The two constraints  of $H_1, H_2$ intersect at point $F=(2t,t)$. The feasible region for (IP2), i.e., the region where $(a, b)$ satisfy both constraints is the lattice points in the shaded area bounded by both constraints and the closed first quadrant excluding the lower envelope. 
Since the objective function has slope in between the slopes given by these two constraints, we know that the objective function (\ref{obj}) of our linear program is maximized at $(2t,t)$ and the objective function (\ref{obji}) of our integer program is maximized at $(2t-1,t+1)$. This corresponds to a graph $G'$ on $2a+b=2(2t-1)+t+1=5t-1$ vertices which is the complement of a matching of size $a=2t-1$. Therefore, among all graphs on $n$ vertices which avoids $H_1$ and $H_2$ as minors, the graph which is a disjoint union copies of $G'$ has nearly the maximum number of cliques, which is  $2^{(2\log_2 3 + 1 + o(1))t}n$.

\section{Concluding remarks} 

We determined the number of cliques in a graph on $n$ vertices with no $K_t$-minor up to a factor $2^{o(t)}$. It would be interesting to determine the exact value, as has been done by Wood \cite{Wo1} for $t \leq 9$. As observed by Wood, for small values of $t$, the extremal example is a graph formed from a $K_{t-2}$ by adding vertices one at a time whose neighorhood in the graph so far is a $K_{t-2}$. Such a graph on $n$ vertices has $2^{t-2}(n-t+3)$ cliques. Wood further conjectures that this bound is tight if and only if $t \geq 49$. It is interesting that for large values of $t$, the extremal graph is very different, coming from the complement of a perfect matching on roughly $4t/3$ vertices instead of from a clique on $t-1$ vertices. It seems plausible that the methods developed here could be useful for solving this problem for large $t$. 

We further extended our result by replacing $K_t$ by any connected graph $H$. It looks interesting and challenging to solve the same problem when $H$ is not connected. We further generalized our result to determine up to a small factor  the maximum number of cliques a graph on $n$ vertices can have in a minor-closed family which is closed under disjoint union. Again, it is an interesting challenge to determine the value exactly or remove the condition that the graph is closed under disjoint union. 

Another interesting problem proposed by Wood \cite{Wo1} is to determine the maximum number of cliques of order $k$ in a $K_t$-minor free graph on $n$ vertices. If $k \geq t$, the answer is clearly $0$. For $k=2$, this asks for the maximum number of edges a $K_t$-minor free graph on $n$ vertices can have, and this extremal problem was asymptotically solved by Thomason \cite{Th1}, where the components of the extremal graphs are pseudorandom of a certain density and order. One might expect Thomason's techniques to extend to the case $k$ is constant. For large $k$, the answer is quite different. The average size of the cliques in the complement of a perfect matching of size $x$ is $2x/3$, and a random clique in this graph typically has about this size. Thus, for $k=4t/9$, the graph which is a complement of a perfect matching of size just less than $2t/3$ is $K_t$-minor free and has nearly the maximum number of $K_k$, namely $3^{2t/3-o(t)}n$. For $k=t-1$ and $n \geq t-1$, Wood \cite{Wo1} shows that the maximum number of $K_k$ is $n-t+2$. We therefore see that depending on the range of $k$, the answer has quite different forms. 

In another paper \cite{FW}, we prove similar results for forbidden immersions or subdivisions. We show that any $n$-vertex graph with no $K_t$-immersion has at most $2^{t+o(t)}n$ cliques, which is tight by the disjoint union of cliques on $t-1$ vertices. We also improve the exponential constant on the number of cliques in a $K_t$-subdivision graph on $n$ vertices but did not determine it. We conjecture that the maximum number of cliques in $K_t$-subdivision free graph on $n$ vertices is $3^{2t/3 +o(t)}n$ as it is when we replace subdivision by minor.

\end{document}